\documentclass[12pt]{amsart}

\usepackage{mathrsfs,amssymb}

\usepackage{booktabs}
\newtheorem{theorem}{Theorem}[section]

\newtheorem{prop}[theorem]{Proposition}

\theoremstyle{definition}

\theoremstyle{remark}
\newtheorem{remark}[theorem]{Remark}

\numberwithin{equation}{section}

\let \la=\lambda
\let \e=\varepsilon

\let \f=\varphi

\let \si=\sigma

\begin{document}
\title[Separated bump conditions]
{On separated bump conditions for Calder\'on-Zygmund operators}

\author[A. K. Lerner]{Andrei K. Lerner}

\address[A. K. Lerner]{Department of Mathematics,
Bar-Ilan University, 5290002 Ramat Gan, Israel}
\email{lernera@math.biu.ac.il}

\thanks{The author was supported by ISF grant No. 447/16 and ERC Starting Grant No. 713927.}

\begin{abstract}
We improve  bump conditions for the two-weight boundedness of Calder\'on-Zygmund operators introduced recently by R. Rahm and S. Spencer \cite{RS}.
\end{abstract}



\maketitle

\section{Introduction}
Let $T$ be a Calder\'on-Zygmund operator on ${\mathbb R}^n$. In this note we are concerned with separated bump  conditions on a couple of weights $(w,\si)$ for which
\begin{equation}\label{twocz}
\|T(f\si)\|_{L^p(w)}\lesssim \|f\|_{L^p(\si)}\quad(1<p<\infty).
\end{equation}

Given a cube $Q\subset {\mathbb R}^n$, denote $f_Q=\frac{1}{|Q|}\int_Qf$. It is well known that the standard $A_p$ condition,
$$
[w,\si]_{A_p}=\sup_Qw_Q(\si_Q)^{p-1}<\infty,
$$
is not sufficient for (\ref{twocz}). In fact, the $A_p$ condition is not sufficient even for the maximal operator instead of $T$ \cite{M}.

A number of works were devoted to finding slightly stronger conditions that are sufficient for (\ref{twocz}). Among such conditions one can distinguish joint and separated bump conditions.
By a joint bump condition we generally mean a condition of the form
$$\sup_QB_1[w;Q](B_2[\si;Q])^{p-1}<\infty,$$
where $B_1[w;Q]$ and $B_2[\si;Q]$ are referred to as bumps, that is, expressions slightly larger than $w_Q$ and $\si_Q$, respectively.
For joint bump conditions, see, e.g., \cite{L,Li,NRTV}.

By a separated bump condition one means a more delicate and weaker condition of the form
$$\sup_QB_1[w;Q](\si_Q)^{p-1}<\infty\quad\text{and}\quad \sup_Qw_Q(B_2[\si;Q])^{p-1}<\infty.$$

There are several different ways of ``bumping", see \cite{CRV,La,Li} for the Orlicz bumps and \cite{LS,RS,TV} for the so-called entropy bumps; we recall them below, in Section 5.

In a recent work by R. Rahm and S. Spencer \cite{RS}, yet another bumps were introduced. Assume that $\f_p$ is a function that is decreasing on $(0,1)$ and increasing on $(1,\infty)$
and such that $\int_0^{\infty}\frac{1}{\f_p(t)^{1/p}}\frac{dt}{t}<\infty$. It was shown in \cite{RS} that if
$$\sup_Qw_Q(\si_Q)^{p-1}\f_p(\si_Q)<\infty\quad\text{and}\quad \sup_Qw_Q(\si_Q)^{p-1}\f_{p'}^{p-1}(w_Q)<\infty,$$
then (\ref{twocz}) holds.

In this note we improve the integrability assumptions on $\f_p$ in the above result. In particular, we will show that for $t\ge 1$, the assumption $\int_1^{\infty}\frac{1}{\f_p(t)^{1/p}}\frac{dt}{t}<\infty$
can be improved to $\int_1^{\infty}\frac{1}{\f_p(t)}\frac{dt}{t}<\infty$. For $0<t<1$ our condition looks a bit technical but it shows that, for example, one can take
$\f_p(t)=\log\big(e+\frac{1}{t}\big)\log\log^{p+\e}\big(e^{e}+\frac{1}{t}\big).$
As in the previous works on this topic, our proof is based on the sparse domination.

The paper is organized as follows. In Section 2 we recall the standard scheme of reducing the initial problem to analysis of testing conditions.
Section 3 contains some, mostly known, auxiliary statements. The main result is contained in Section~4. In Section 5 we give a brief overview of known bump conditions.
Section 6 contains some further remarks and complements.

Throughout the paper we use the notation $A\lesssim B$ if $A\le CB$ with some independent constant $C$. We write $A\simeq B$ if $A\lesssim B$ and $B\lesssim A$.

\section{Standard reductions}
As in most of the previous works dealing with bump conditions, we will make use of the following tools.

\begin{list}{\labelitemi}{\leftmargin=1em}
\item Reducing to sparse operators. Recall that the sparse operator $A_{\mathcal S}$ is defined by
$$A_{\mathcal S}f=\sum_{Q\in {\mathcal S}}f_Q\chi_Q,$$
where ${\mathcal S}$ is a sparse family of dyadic cubes, which means that there exist disjoint subsets $E_Q\subset Q\in {\mathcal S}$ such that $|E_Q|\simeq |Q|$.
Since a Calder\'on-Zygmund operator~$T$ is pointwise bounded by at most $3^n$ sparse operators $A_{\mathcal S}$ (see, e.g., \cite{LO} for a very short proof of this fact), the problem is reduced
to finding sufficient conditions for $A_{\mathcal S}$ instead of $T$ in (\ref{twocz}).

\item In turn, the two-weight boundedness for $A_{\mathcal S}$ is characterized by testing conditions. Denote by $[w,\si]_p$ the best possible constant such that
for every $R\in {\mathcal S}$,
\begin{equation}\label{ws}
\Big\|\sum_{Q\in {\mathcal S}, Q\subseteq R}\si_Q\chi_Q\Big\|_{L^p(w)}\le [w,\si]_p\si(R)^{1/p}.
\end{equation}
Then (see \cite{H,LSU,T})
$$\|A_{{\mathcal S}}(\cdot\si)\|_{L^p(\si)\to L^p(w)}\simeq [w,\si]_p+[\si,w]_{p'}.$$

Thus, separated bump conditions typically appear as conditions on $(w,\si)$ for which $[w,\si]_p$ and $[\si,w]_{p'}$ are finite. By symmetry, it suffices to analyze $[w,\si]_p$.

\item As it was shown in \cite{HL, Li}, a very efficient way to handle the left-hand side of (\ref{ws}) is based on the following facts established respectively in \cite{COV} and \cite{H}:
for every dyadic lattice ${\mathscr D}$ and any non-negative sequence $\{a_Q\}_{Q\in {\mathscr D}}$,
\begin{equation}\label{cov}
\Big\|\sum_{Q\in {\mathscr D}}a_Q\chi_Q\Big\|_{L^p(w)}\simeq \left(\sum_{Q\in {\mathscr D}}a_Q\Big(\frac{1}{w(Q)}\sum_{Q'\in {\mathscr D}, Q'\subseteq Q}a_{Q'}w(Q')\Big)^{p-1}w(Q)\right)^{1/p},
\end{equation}
and for every sparse family ${\mathcal S}$ and $0<s<1$,
\begin{equation}\label{h}
\sum_{Q\in {\mathcal S}, Q\subseteq R}(w_Q)^{s}|Q|\lesssim (w_R)^{s}|R|.
\end{equation}

\item It was proved in \cite{H} that for every sparse family of dyadic cubes ${\mathcal S}$,
\begin{equation}\label{hyt}
\int_R\Big(\sum_{Q\in {\mathcal S}, Q\subseteq R}\si_Q\chi_Q\Big)^pw\lesssim (\sup_{Q\in{\mathcal S}}w_Q\si_Q^{p-1})\sum_{Q\in {\mathcal S},Q\subseteq R}\si(Q)
\end{equation}
(observe that this can be shown by combining (\ref{cov}) and (\ref{h})).
\end{list}

\section{Auxiliary propositions}
The following result is based on the same ideas as in \cite{La,Li}.

\begin{prop}\label{mp} Given a weight $\si$ and a cube $Q$, define $\la_Q(\si)$ such that $\la_Q(\si)\ge 1$ and for every sparse family ${\mathcal S}$,
\begin{equation}\label{conds}
\sum_{Q\in {\mathcal S}, Q\subseteq R}\la_Q(\si)^{-1}\si(Q)\lesssim \si(R).
\end{equation}
Then
$$[w,\si]_p\lesssim \sup_{Q\in {\mathcal S}}(w_Q)^{1/p}(\si_Q)^{1/p'}\la_Q(\si)^{1/p}\f^{1/p'}(\la_Q(\si)),$$
where $\f$ is an increasing function such that $\int_1^{\infty}\frac{1}{\f(t)}\frac{dt}{t}<\infty$.
\end{prop}

\begin{proof} For $k\ge 0$ define
\begin{equation}\label{fk}
{\mathcal F}_k=\{Q\in {\mathcal S}: 2^k\le \la_Q(\si)\le 2^{k+1}\}.
\end{equation}
Then, setting $M=\sup_{Q\in {\mathcal S}}w_Q(\si_Q)^{p-1}\la_Q(\si)\f^{p-1}(\la_Q(\si))$, by (\ref{hyt}) we obtain
\begin{eqnarray}
&&\Big\|\sum_{Q\in {\mathcal S}, Q\subseteq R}\si_Q\chi_Q\Big\|_{L^p(w)}\le \sum_{k=0}^{\infty}\Big\|\sum_{Q\in {\mathcal F}_k, Q\subseteq R}\si_Q\chi_Q\Big\|_{L^p(w)}\nonumber\\
&&\lesssim \sum_{k=0}^{\infty}\Big((\sup_{Q\in{\mathcal F}_k}w_Q\si_Q^{p-1})\sum_{Q\in {\mathcal F}_k,Q\subseteq R}\si(Q)\Big)^{1/p}\nonumber\\
&&\lesssim M^{1/p}\sum_{k=0}^{\infty}\frac{1}{\f(2^k)^{1/p'}}\Big(\sum_{Q\in {\mathcal F}_k,Q\subseteq R}\la_Q(\si)^{-1}\si(Q)\Big)^{1/p}\label{subs}\\
&&\lesssim M^{1/p}\Big(\sum_{k=0}^{\infty}\frac{1}{\f(2^k)}\Big)^{1/p'}\Big(\sum_{Q\in {\mathcal S}, Q\subseteq R}\la_Q(\si)^{-1}\si(Q)\Big)^{1/p}\lesssim M^{1/p}\si(R)^{1/p},\nonumber
\end{eqnarray}
which completes the proof.
\end{proof}

The following proposition is contained in \cite{RS}. We give its proof for the sake of completeness.

\begin{prop}\label{rs} Let ${\mathcal S}$ be a sparse family of dyadic cubes. For $k\in {\mathbb Z}$ define
$${\mathcal F}_k=\{Q\in {\mathcal S}: 2^k<\si_Q\le 2^{k+1}\}.$$
Then
$$\sum_{Q\in {\mathcal F}_k,Q\subseteq R}\si(Q)\lesssim \si(R).$$
\end{prop}

\begin{proof} Let $\{Q_j\}$ be the maximal cubes of $\{Q\in {\mathcal F}_k, Q\subseteq R\}$. Then, by maximality, they are pairwise disjoint, and also, by sparseness,
\begin{eqnarray*}
\sum_{Q\in {\mathcal F}_k,Q\subseteq R}\si(Q)&=&\sum_j\sum_{Q\in {\mathcal F}_k,Q\subseteq Q_j}\si(Q)\le 2^{k+1}\sum_j\sum_{Q\in {\mathcal F}_k,Q\subseteq Q_j}|Q|\\
&\lesssim& 2^k\sum_j|Q_j|\lesssim \sum_j\si(Q_j)\lesssim \si(R),
\end{eqnarray*}
which completes the proof.
\end{proof}

\begin{prop}\label{psi} Let $\psi$ be a function that is decreasing on $(0,1)$ and increasing on $(1,\infty)$, and such that $\int_0^{\infty}\frac{1}{\psi(t)}\frac{dt}{t}<\infty$.
Then
$$\sum_{Q\in {\mathcal S},Q\subseteq R}\frac{1}{\psi(\si_Q)}\si(Q)\lesssim \si(R).$$
\end{prop}

\begin{proof} Setting ${\mathcal F}_k$ as in the previous proposition, we obtain
$$
\sum_{Q\in {\mathcal S},Q\subseteq R}\frac{1}{\psi(\si_Q)}\si(Q)=\sum_{k\in {\mathbb Z}}\sum_{Q\in {\mathcal F}_k,Q\subseteq R}\frac{1}{\psi(\si_Q)}\si(Q)\lesssim
\Big(\sum_{k\in {\mathbb Z}}\frac{1}{\psi(2^k)}\Big)\si(R),
$$
and we are done.
\end{proof}

\section{Main result}
\begin{theorem}\label{mr}
Assume that $\psi$ is a function that is decreasing on $(0,1)$ and increasing on $(1,\infty)$, and such that $\int_0^{\infty}\frac{1}{\psi(t)}\frac{dt}{t}<\infty$. Assume also that
$\psi(t)\lesssim e^{\sqrt t}$ for $t\ge 1$. Next, let $\f$ be an increasing function on $(1,\infty)$ such that $\int_1^{\infty}\frac{1}{\f(t)}\frac{dt}{t}<\infty$. Define
$$
\nu_p(t) = \begin{cases} \psi(t)\f^{p-1}(\psi(t)), &0<t<1\\
 \psi(t),& t\ge 1, \end{cases}
$$
and set
$$[w,\si]_{\nu_p}=\sup_{Q}w_Q(\si_Q)^{p-1}\nu_p(\si_Q).$$
Then
$$\|T(\cdot\si)\|_{L^p(\si)\to L^p(w)}\lesssim [w,\si]_{\nu_p}^{1/p}+[\si,w]_{\nu_{p'}}^{1/p'}.$$
\end{theorem}

\begin{remark}\label{r1} A typical example of $\nu_p$ on $(0,1)$ can be obtained by setting
$$\psi(t)=\log\Big(e+\frac{1}{t}\Big)\log\log^{1+\e}\Big(e^e+\frac{1}{t}\Big)\quad(0<t<1)$$
and
$$\f(t)=\log(e+t)\log\log^{1+\e}(e^e+t)\quad(1<t<\infty).$$
Then, for $0<t<1$,
$$
\nu_p(t)\simeq
\log\Big(e+\frac{1}{t}\Big)\log\log^{p+\e}\Big(e^e+\frac{1}{t}\Big)\log\log\log^{(p-1)(1+\e)}\Big(e^{e^e}+\frac{1}{t}\Big).
$$
\end{remark}

\begin{remark}\label{r2} In the cases of interest the growth of  $\psi$ for $t\ge 1$ is logarithmic (e.g., $\psi(t)=\log^{1+\e}(e+t)$
or $\psi(t)=\log(e+t)\log\log^{1+\e}(e^e+t)$, etc.), and hence the assumption
$\psi(t)\lesssim e^{\sqrt t}$ for $t\ge 1$ holds trivially.
\end{remark}

\begin{proof}[Proof of Theorem \ref{mr}] As we have discussed in Section 2, it suffices to estimate $[w,\si]_p$ in (\ref{ws}), that is, our goal is to show that
\begin{equation}\label{goal}
\Big\|\sum_{Q\in {\mathcal S}, Q\subseteq R}\si_Q\chi_Q\Big\|_{L^p(w)}\lesssim [w,\si]_{\nu_p}^{1/p}\si(R)^{1/p}.
\end{equation}

Set
$${\mathcal S}_1=\{Q\in {\mathcal S}:\si_Q<1\}\quad\text{and}\quad {\mathcal S}_2=\{Q\in {\mathcal S}:\si_Q\ge 1\}.$$
An immediate combination of Propositions \ref{mp} and \ref{psi} yields
\begin{eqnarray*}
\Big\|\sum_{Q\in {\mathcal S}_1, Q\subseteq R}\si_Q\chi_Q\Big\|_{L^p(w)}
&\lesssim& \Big(\sup_{Q\in {\mathcal S}_1}(w_Q)^{1/p}(\si_Q)^{1/p'}\psi(\si_Q)^{1/p}\f^{1/p'}(\psi(\si_Q))\Big)\si(R)^{1/p}\\
&\lesssim& [w,\si]_{\nu_p}^{1/p}\si(R)^{1/p}.
\end{eqnarray*}

Therefore, in order to prove (\ref{goal}), it remains to show that
\begin{equation}\label{rem}
\Big\|\sum_{Q\in {\mathcal S}_2, Q\subseteq R}\si_Q\chi_Q\Big\|_{L^p(w)}^p\lesssim [w,\si]_{\nu_p}\si(R).
\end{equation}
At this point we apply (\ref{cov}), which says that
\begin{equation}\label{inn}
\Big\|\sum_{Q\in {\mathcal S}_2, Q\subseteq R}\si_Q\chi_Q\Big\|_{L^p(w)}^p\simeq \sum_{Q\in {\mathcal S}_2, Q\subseteq R}\si_Q\Big(\frac{1}{w(Q)}
\sum_{Q'\in {\mathcal S}_2, Q'\subseteq Q}\si_{Q'}w(Q')\Big)^{p-1}w(Q).
\end{equation}
Split the inner sum on the right-hand side of~(\ref{inn}) as follows:
\begin{equation}\label{nspl}
\sum_{Q'\in {\mathcal S}_2, Q'\subseteq Q}\si_{Q'}w(Q')=
\sum_{Q'\in {\mathcal S}_2, Q'\subseteq Q\atop \si_{Q'}\le \si_Q^{1/2}}\si_{Q'}w(Q')+
\sum_{Q'\in {\mathcal S}_2, Q'\subseteq Q\atop \si_{Q'}>\si_Q^{1/2}}\si_{Q'}w(Q').
\end{equation}

Suppose first that $p\ge 2$.
Consider the first sum on the right-hand side of (\ref{nspl}). Denote
$${\mathcal F}_0=\{Q'\in {\mathcal S}_2: w_{Q'}(\si_{Q'})^{p-1}\le [w,\si]_{A_p}\psi(\si_Q)^{-1}\}$$
and, for $k=1,\dots,N\simeq \log\big(\psi(\si_Q)\big)$,
$${\mathcal F}_k=\{Q'\in {\mathcal S}_2: 2^{k-1}[w,\si]_{A_p}\psi(\si_Q)^{-1}<w_{Q'}(\si_{Q'})^{p-1}\le 2^{k}[w,\si]_{A_p}\psi(\si_Q)^{-1}\},$$
where, abusing the notation, we set
$$[w,\si]_{A_p}=\sup_{Q\in {\mathcal S}_2}w_Q(\si_Q)^{p-1}.$$
Denote also
$$E_k=\{Q'\in {\mathcal F}_k: Q'\subseteq Q, \si_{Q'}\le \si_Q^{1/2}\}.$$
Then
\begin{eqnarray*}
\sum_{{Q'\in {\mathcal S}_2, Q'\subseteq Q}\atop \si_{Q'}\le \si_Q^{1/2}}\si_{Q'}w(Q')=\sum_{k=0}^N\sum_{Q'\in E_k}\si_{Q'}w(Q').\\
\end{eqnarray*}

By (\ref{hyt}),
\begin{eqnarray*}
\sum_{Q'\in E_0}\si_{Q'}w(Q')
&\le&\left(\frac{[w,\si]_{A_p}}{\psi(\si_Q)}\right)^{\frac{1}{p-1}}\sum_{{Q'\in {\mathcal S}_2, Q'\subseteq Q}}(w_{Q'})^{1-\frac{1}{p-1}}|Q'|\\
&\lesssim &\left(\frac{[w,\si]_{A_p}}{\psi(\si_Q)}\right)^{\frac{1}{p-1}}(w_{Q})^{1-\frac{1}{p-1}}|Q|.
\end{eqnarray*}
Fix $1\le k\le N$. Let $\{Q_j\}$ be the maximal cubes of $E_k$.
Then, by (\ref{hyt}),
\begin{eqnarray*}
\sum_{Q'\in E_k,Q'\subseteq Q_j}\si_{Q'}w(Q')
&\le&\left(\frac{2^{k}[w,\si]_{A_p}}{\psi(\si_Q)}\right)^{\frac{1}{p-1}}\sum_{Q'\in E_k,Q'\subseteq Q_j}(w_{Q'})^{1-\frac{1}{p-1}}|Q'|\\
&\lesssim& \left(\frac{2^{k}[w,\si]_{A_p}}{\psi(\si_Q)}\right)^{\frac{1}{p-1}}(w_{Q_j})^{1-\frac{1}{p-1}}|Q_j|\\
&\lesssim& \si_{Q_j}(w_{Q_j})^{\frac{1}{p-1}}(w_{Q_j})^{1-\frac{1}{p-1}}|Q_j|\lesssim \si_Q^{1/2}w(Q_j).
\end{eqnarray*}
Hence, since $\{Q_j\}$ are pairwise disjoint,
$$
\sum_{Q'\in E_k}\si_{Q'}w(Q')=\sum_j\sum_{Q'\in E_k,Q'\subseteq Q_j}\si_{Q'}w(Q')
\lesssim \si_Q^{1/2}w(Q).
$$
From this, and using also that, by our assumption, $\log\big(\psi(\si_Q)\big)\lesssim \si_Q^{1/2}$, we obtain
$$\sum_{k=1}^N\sum_{Q'\in E_k}\si_{Q'}w(Q')\lesssim \log\big(\psi(\si_Q)\big)\si_Q^{1/2}w(Q)\lesssim \si_Qw(Q).$$

Collecting the above estimates yields
$$\sum_{Q'\in {\mathcal S}_2, Q'\subseteq Q\atop \si_{Q'}\le \si_Q^{1/2}}\si_{Q'}w(Q')\lesssim
\left(\frac{[w,\si]_{A_p}}{\psi(\si_Q)}\right)^{\frac{1}{p-1}}(w_{Q})^{1-\frac{1}{p-1}}|Q|+\si_Qw(Q).$$
Therefore, by Proposition \ref{psi},
\begin{eqnarray}
&&\sum_{Q\in {\mathcal S}_2, Q\subseteq R}\si_Q\Big(\frac{1}{w(Q)}\sum_{Q'\in {\mathcal S}_2, Q'\subseteq Q\atop \si_{Q'}\le \si_Q^{1/2}}\si_{Q'}w(Q')\Big)^{p-1}w(Q)\label{part}\\
&&\lesssim [w,\si]_{A_p}\sum_{Q\in {\mathcal S}_2, Q\subseteq R}\frac{1}{\psi(\si_Q)}\si(Q)+\sum_{Q\in {\mathcal S}_2, Q\subseteq R}\si_Q^{p}w(Q)\nonumber\\
&&\lesssim [w,\si]_{\nu_p}\sum_{Q\in {\mathcal S}_2, Q\subseteq R}\frac{1}{\psi(\si_Q)}\si(Q)\lesssim [w,\si]_{\nu_p}\si(R).\nonumber
\end{eqnarray}

Consider the second sum on the right-hand side of (\ref{nspl}). Since $\psi$ is increasing on $(1,\infty)$,
\begin{eqnarray*}
\sum_{{Q'\in {\mathcal S}_2, Q'\subseteq Q}\atop \si_{Q'}>\si_Q^{1/2}}\si_{Q'}w(Q')
&\le& \psi(\si_Q^{1/2})^{-\frac{1}{p-1}}
\sum_{{Q'\in {\mathcal S}_2, Q'\subseteq Q}\atop \si_{Q'}>\si_Q^{1/2}}\si_{Q'}w(Q')\psi(\si_{Q'})^{\frac{1}{p-1}}\\
&\le& [w,\si]_{\nu_p}^{\frac{1}{p-1}}\psi(\si_Q^{1/2})^{-\frac{1}{p-1}}\sum_{{Q'\in {\mathcal S}_2, Q'\subseteq Q}}(w_{Q'})^{1-\frac{1}{p-1}}|Q'|\\
&\lesssim& [w,\si]_{\nu_p}^{\frac{1}{p-1}}\psi(\si_Q^{1/2})^{-\frac{1}{p-1}}(w_{Q})^{1-\frac{1}{p-1}}|Q|.\\
\end{eqnarray*}
Hence,
\begin{eqnarray*}
&&\sum_{Q\in {\mathcal S}_2, Q\subseteq R}\si_Q\Big(\frac{1}{w(Q)}\sum_{Q'\in {\mathcal S}_2, Q'\subseteq Q\atop \si_{Q'}>\si_Q^{1/2}}\si_{Q'}w(Q')\Big)^{p-1}w(Q)\\
&&\lesssim [w,\si]_{\nu_p}\sum_{Q\in {\mathcal S}_2, Q\subseteq R}\frac{1}{\psi(\si_Q^{1/2})}\si(Q)\lesssim [w,\si]_{\nu_p}\si(R),\\
\end{eqnarray*}
where we have used again Proposition \ref{psi} and that $\int_1^{\infty}\frac{1}{\psi(t^{1/2})}\frac{dt}{t}<\infty.$
This, along with (\ref{part}), completes the proof of (\ref{rem}) in the case $p\ge 2$.

In the case $1<p<2$ the proof is similar. Consider the first sum on the right-hand side of (\ref{nspl}).
Define the sets $E_k$ exactly as in the previous case. By (\ref{hyt}),
\begin{eqnarray*}
\sum_{Q'\in E_0}\si_{Q'}w(Q')
&\le&\frac{[w,\si]_{A_p}}{\psi(\si_Q)}\sum_{{Q'\in {\mathcal S}_2, Q'\subseteq Q}}\si_{Q'}^{2-p}|Q'|\\
&\lesssim & \frac{[w,\si]_{A_p}}{\psi(\si_Q)}\si_{Q}^{2-p}|Q|.
\end{eqnarray*}
Fix $1\le k\le N$. Let $\{Q_j\}$ be the maximal cubes of $E_k$.
Then
\begin{eqnarray*}
&&\sum_{Q'\in E_k,Q'\subseteq Q_j}\si_{Q'}w(Q')\le\frac{2^{k}[w,\si]_{A_p}}{\psi(\si_Q)}\sum_{Q'\in E_k,Q'\subseteq Q_j}\si_{Q'}^{2-p}|Q'|\\
&&\lesssim \frac{2^{k}[w,\si]_{A_p}}{\psi(\si_Q)}\si_{Q_j}^{2-p}|Q_j|\lesssim w_{Q_j}\si_{Q_j}^{p-1}\si_{Q_j}^{2-p}|Q_j|\lesssim \si_Q^{1/2}w(Q_j).
\end{eqnarray*}
Therefore,
$$
\sum_{Q'\in E_k}\si_{Q'}w(Q')=\sum_{j}\sum_{Q'\in E_k,Q'\subseteq Q_j}\si_{Q'}w(Q')\lesssim \si_Q^{1/2}w(Q),
$$
which implies
\begin{eqnarray*}
\sum_{{Q'\in {\mathcal S}_2, Q'\subseteq Q}\atop \si_{Q'}\le \si_Q^{1/2}}\si_{Q'}w(Q')&\lesssim& \frac{[w,\si]_{A_p}}{\psi(\si_Q)}\si_{Q}^{2-p}|Q|+\log\big(\psi(\si_Q)\big)\si_Q^{1/2}w(Q)\\
&\lesssim& \frac{[w,\si]_{A_p}}{\psi(\si_Q)}\si_{Q}^{2-p}|Q|+\si_Qw(Q).
\end{eqnarray*}
Hence, by Proposition \ref{psi},
\begin{eqnarray}
&&\sum_{Q\in {\mathcal S}_2, Q\subseteq R}\si_Q\Big(\frac{1}{w(Q)}\sum_{Q'\in {\mathcal S}_2, Q'\subseteq Q\atop \si_{Q'}\le \si_Q^{1/2}}\si_{Q'}w(Q')\Big)^{p-1}w(Q)\label{part2}\\
&&\lesssim [w,\si]_{A_p}^{p-1}\sum_{Q\in {\mathcal S}_2, Q\subseteq R}\frac{1}{\psi^{p-1}(\si_Q)}\big((\si_Q)^{p-1}w_Q\big)^{2-p}\si(Q)+\sum_{Q\in {\mathcal S}_2, Q\subseteq R}\si_Q^{p}w(Q)\nonumber\\
&&\lesssim [w,\si]_{\nu_p}\sum_{Q\in {\mathcal S}_2, Q\subseteq R}\frac{1}{\psi(\si_Q)}\si(Q)\lesssim [w,\si]_{\nu_p}\si(R).\nonumber
\end{eqnarray}

Further, arguing as above,
\begin{eqnarray*}
\sum_{{Q'\in {\mathcal S}_2, Q'\subseteq Q}\atop \si_{Q'}>\si_Q^{1/2}}\si_{Q'}w(Q')&\le& \psi(\si_Q^{1/2})^{-1}
\sum_{{Q'\in {\mathcal S}_2, Q'\subseteq Q}\atop \si_{Q'}>\si_Q^{1/2}}\si_{Q'}w(Q')\psi(\si_{Q'})\\
&\lesssim &\frac{[w,\si]_{\nu_p}}{\psi(\si_Q^{1/2})}\si_Q^{2-p}|Q|.
\end{eqnarray*}
Therefore,
\begin{eqnarray*}
&&\sum_{Q\in {\mathcal S}_2, Q\subseteq R}\si_Q\Big(\frac{1}{w(Q)}\sum_{Q'\in {\mathcal S}_2, Q'\subseteq Q\atop \si_{Q'}>\si_Q^{1/2}}\si_{Q'}w(Q')\Big)^{p-1}w(Q)\\
&&\lesssim [w,\si]_{\nu_p}^{p-1}\sum_{Q\in {\mathcal S}_2, Q\subseteq R}\frac{1}{\psi^{p-1}(\si_Q^{1/2})}\big((\si_Q)^{p-1}w_Q\big)^{2-p}\si(Q)\\
&&\lesssim [w,\si]_{\nu_p}\sum_{Q\in {\mathcal S}_2, Q\subseteq R}\frac{1}{\psi(\si_Q^{1/2})}\si(Q)\lesssim [w,\si]_{\nu_p}\si(R),
\end{eqnarray*}
which, along with (\ref{part2}), proves (\ref{rem}) in the case $1<p<2$. This completes the proof.
\end{proof}

\section{A brief survey of different bump conditions} The approach described in Proposition \ref{mp} is the key to different bump conditions, and we overview them briefly.
\subsection{Orlicz bumps}
Recall that for a Young function $A$, the normalized Luxemburg norm is defined by
$$
\|f\|_{A,Q}=\inf\Big\{\la>0:\frac{1}{|Q|}\int_QA(|f(y)|/\la)dy\le 1\Big\}.
$$
Define the maximal operator $M_A$ by
$$M_Af(x)=\sup_{Q\ni x}\|f\|_{A,Q}.$$

We say that a Young function $A$ satisfies the $B_p$ condition if $\int_1^{\infty}\frac{A(t)}{t^p}\frac{dt}{t}<\infty$.
Assume that $A\in B_p$, and set $\la_Q(\si)=\frac{\si_Q}{\|\si^{1/p}\|_{A,Q}^p}$. Then
\begin{eqnarray*}
&&\sum_{Q\in {\mathcal S},Q\subseteq R}\la_Q(\si)^{-1}\si(Q)=\sum_{Q\in {\mathcal S},Q\subseteq R}\|\si^{1/p}\|_{A,Q}^p|Q|\\
&&\lesssim \sum_{Q\in {\mathcal S},Q\subseteq R}\int_{E_Q}(M_A(\si^{1/p}\chi_R))^p \lesssim \int_{R}(M_A(\si^{1/p}\chi_R))^p \lesssim \si(R),
\end{eqnarray*}
where we have used that $M_A$ is bounded on $L^p$ for $A\in B_p$ \cite{P}.
Hence, by Proposition \ref{mp},
\begin{equation}\label{kli}
[w,\si]_p\lesssim \sup_Q(w_Q)^{1/p}\frac{\si_Q}{\|\si^{1/p}\|_{A,Q}}\f^{1/p'}\left(\frac{\si_Q}{\|\si^{1/p}\|_{A,Q}^p}\right).
\end{equation}
This result was obtained by K. Li \cite{Li}.

Let $\bar A$ denote the Young function complementary to $A$.
By generalized H\"older's inequality,
$$
\si_Q\le 2\|\si^{1/p}\|_{A,Q}\|\si^{1/p'}\|_{\bar A, Q}.
$$
From this and from (\ref{kli}),
\begin{equation}\label{la}
[w,\si]_p\lesssim \sup_Q(w_Q)^{1/p}\|\si^{1/p'}\|_{\bar A, Q}\f^{1/p'}\left(\frac{\|\si^{1/p'}\|_{\bar A,Q}^p}{(\si_Q)^{p-1}}\right).
\end{equation}
This result was obtained by M. Lacey \cite{La}.

\subsection{Entropy bumps}
Assume that instead of (\ref{conds}) we have
$$
\sum_{Q\in {\mathcal F}_k, Q\subseteq R}\si(Q)\lesssim 2^k\si(R),
$$
where the sets ${\mathcal F}_k$ are defined by (\ref{fk}). In this case, setting
$$M=\sup_Qw_Q(\si_Q)^{p-1}\la_Q(\si)\f^{p}(\la_Q(\si)),$$
instead of (\ref{subs}) we obtain
$$M^{1/p}\sum_{k=0}^{\infty}\frac{1}{\f(2^k)}\Big(\sum_{Q\in {\mathcal F}_k,Q\subseteq R}\la_Q(\si)^{-1}\si(Q)\Big)^{1/p}\lesssim M^{1/p}\si(R)^{1/p},$$
which implies
\begin{equation}\label{entb}
[w,\si]_p\lesssim \sup_Q(w_Q)^{1/p}(\si_Q)^{1/p'}\la_Q(\si)^{1/p}\f(\la_Q(\si)).
\end{equation}

Denote $\la_Q(\si)=\frac{\int_QM(\si\chi_Q)}{\si(Q)}$, and let us consider the sets ${\mathcal F}_k$ defined by (\ref{fk}).
Let $\{Q_j\}$ be the maximal cubes of $\{Q\in {\mathcal F}_k, Q\subseteq R\}$. Then
\begin{eqnarray*}
\sum_{Q\in {\mathcal F}_k, Q\subseteq R}\si(Q)=\sum_j\sum_{Q\in {\mathcal F}_k, Q\subseteq Q_j}\si(Q)\lesssim \sum_j\int_{Q_j}M(\si\chi_{Q_j})\lesssim
2^k\sum_j\si(Q_j)\lesssim 2^k\si(R).
\end{eqnarray*}

Therefore, by (\ref{entb}),
\begin{equation}\label{entropy}
[w,\si]_p\lesssim \sup_Q(w_Q)^{1/p}(\si_Q)^{1/p'}\left(\frac{\int_QM(\si\chi_Q)}{\si(Q)}\right)^{1/p}\f\left(\frac{\int_QM(\si\chi_Q)}{\si(Q)}\right),
\end{equation}
where $\int_1^{\infty}\frac{1}{\f(t)}\frac{dt}{t}<\infty$.
In the case $p=2$ this result was obtained by S. Treil and A.~Volberg \cite{TV} (who gave the name ``entropy bumps" to the bumps appearing in this expression),
and it was extended to any $p>1$ by M. Lacey and S. Spencer \cite{LS} (see also \cite{RS}).

\section{Remarks and complements}
\subsection{Comparison between different bump conditions}
Although we do not give concrete examples, it is not difficult to see that among bump conditions mentioned
above, there is no universally better condition.

For example, the entropy bump condition appearing in (\ref{entropy}) requires that $\si$ belongs locally to $L\log L$, while in (\ref{goal}) and (\ref{kli}) only local integrability of $\si$ is assumed.

The difference between (\ref{goal}) and (\ref{kli}) is expressed in the difference between
$$\la_Q(\si)=\psi(\si_Q)\quad\text{and}\quad \la_Q(\si)=\frac{\si_Q}{\|\si^{1/p}\|_{A,Q}^p}.$$
On the one hand, by homogeneity, $\psi(\si_Q)$ can not be estimated by $\frac{\si_Q}{\|\si^{1/p}\|_{A,Q}^p}$. On the other hand, for $A\in B_p$, it is easy to find a sequence $\si_j$ such that
$(\si_j)_Q=1$ and $\|\si_j^{1/p}\|_{A,Q}\to 0$ as $j\to \infty$ (it suffices to consider $\si=\frac{|Q|}{|E|}\chi_E$ for $E\subset Q$), and therefore, $\frac{\si_Q}{\|\si^{1/p}\|_{A,Q}^p}$ can not be estimated by
$\psi(\si_Q)$.

Concerning practical applications, the condition in Theorem \ref{mr} is the simplest as it basically requires only the computation of $w_Q$ and $\si_Q$. To check (\ref{kli}), one should estimate $\|\si^{1/p}\|_{A,Q}$
from below, which is a more difficult task.

\subsection{A new two-weight bound for the maximal operator}
Let $M$ denote the Hardy-Littlewood maximal operator. Using Sawyer's two-weight characterization for $M$ \cite{S}, it is easy to show that
$$\|M(\cdot\si)\|_{L^p(\si)\to L^p(w)}\lesssim \left(\sup_R\frac{1}{\si(R)}
\sum_{Q\in {\mathcal S}, Q\subseteq R}\si_Q^pw(Q)\right)^{1/p}.
$$

Let $\la_Q(\si)$ satisfy (\ref{conds}). Then
\begin{equation}\label{th}
\sum_{Q\in {\mathcal S}, Q\subseteq R}\si_Q^pw(Q)\lesssim \big(\sup_{Q}w_Q(\si_Q)^{p-1}\la_Q(\si)\big)\si(R).
\end{equation}
Therefore,
$$
\|M(\cdot\si)\|_{L^p(\si)\to L^p(w)}\lesssim \sup_{Q}w_Q^{1/p}\si_Q^{1/p'}\la_Q(\si)^{1/p}.
$$
Combining this with Proposition \ref{psi} yields
\begin{equation}\label{nb}
\|M(\cdot\si)\|_{L^p(\si)\to L^p(w)}\lesssim \sup_{Q}w_Q^{1/p}\si_Q^{1/p'}\psi(\si_Q)^{1/p}.
\end{equation}
This bound seems to be new.

\subsection{On the separated bump conjecture}
The separated bump conjecture (probably first formulated in \cite{CRV} in a slightly different form) asserts that if $A\in B_p$, then
\begin{equation}\label{sepcon}
[w,\si]_p\lesssim \sup_Q(w_Q)^{1/p}\|\si^{1/p'}\|_{\bar A,Q}.
\end{equation}
This conjecture is still open. In the particular case when $A(t)=\frac{t^p}{\log^{1+\e}(e+t)}$ it was confirmed in \cite{CRV,HP,La};
in general, (\ref{kli}) and (\ref{la}) represent the currently best known bounds towards this conjecture.

Informally speaking, the idea behind the separated bump conjecture is that a ``good" upper bound for $\|M(\cdot\si)\|_{L^p(\si)\to L^p(w)}$ should also be an upper bound for $[w,\si]_p$. Having this point of view in mind and taking into
account (\ref{nb}), one can also conjecture that
\begin{equation}\label{nc}
[w,\si]_p\lesssim \sup_{Q}w_Q^{1/p}\si_Q^{1/p'}\psi(\si_Q)^{1/p},
\end{equation}
where $\psi$ satisfies the assumptions of Proposition \ref{psi}. Theorem \ref{mr} shows that this conjecture holds on the set $\{Q:\si_Q>1\}$.

One can also ask a weaker question whether the finiteness of the right-hand side of (\ref{sepcon}) or (\ref{nc}) implies that $[w,\si]_p<\infty$.
At this point, we make an elementary observation that if $E=\{Q:w_Q(\si_Q)^{p-1}\ge 1\}$, then, by(\ref{hyt}),
\begin{eqnarray*}
\int_R\Big(\sum_{Q\in {\mathcal S}\cap E, Q\subseteq R}\si_Q\chi_Q\Big)^pw&\lesssim& (\sup_{Q\in{\mathcal S}\cap E}w_Q\si_Q^{p-1})\sum_{Q\in {\mathcal S}\cap E,Q\subseteq R}\si(Q)\\
&\lesssim& [w,\si]_{A_p}\sum_{Q\in {\mathcal S},Q\subseteq R}\si_Q^pw(Q).
\end{eqnarray*}
This along with (\ref{th}) shows that in order to get a counterexample (if exists) to such a weaker question, the principal role should be played by cubes $Q$ with $w_Q(\si_Q)^{p-1}<1$.

\end{document}